\documentclass[a4paper, 11pt]{article}

\usepackage{amsmath}
\usepackage{amsthm}
\usepackage{amsfonts}
\usepackage{amssymb}
\usepackage{amscd}
\usepackage[dvips]{graphicx, color}
\newtheorem{thm}{Theorem}[section]
\newtheorem{cor}[thm]{Corollary}

\newtheorem{lem}[thm]{Lemma}
\newtheorem{prop}[thm]{Proposition}
\newtheorem{rem}[thm]{Remark}

\numberwithin{equation}{section}

\newcommand{\Mod}{\text{\boldmath{$\mathsf{M}$}}}

\begin{document}
\title{\bf \Large Reconstruction of weak bialgebra maps and its applications}
\author{
{Michihisa Wakui}\\ 
{\footnotesize  Department of Mathematics, Faculty of Engineering Science,}\\ 
{\footnotesize  Kansai University, Suita-shi, Osaka 564-8680, Japan}\\
{\footnotesize  E-mail address: wakui@kansai-u.ac.jp}\\ 
}
\date{}

\maketitle 
\begin{abstract}
In this note we give a precise statement and a detailed proof for reconstruction problem of weak bialgebra maps. 
As an application we characterize indecomposability of weak algebras in categorical setting. 
\end{abstract}

\baselineskip 15pt 

\section{Introduction}
\par 
The concept of weak bialgebras and weak Hopf algebras was introduced by B\"{o}hm, Nill and Szlach\'{a}nyi \cite{BNS} as a generalization of bialgebras and Hopf algebras. 
In some literatures they are called quantum groupoids \cite{NTV}. 
As shown by Etingof, Nikshych and Ostrik \cite{ENO} the language of weak Hopf algebras is convenient to visualize various categorical construction, including (multi-)fusion categories. 
At present, a lot of concepts and results for the ordinal bialgebras and Hopf algebras are generalized or extended to weak versions. 
\par 
By the classical Tannaka-Krein reconstruction theorem, 
it is well known that any finite-dimensional bialgebra $A$ over a fixed field $\boldsymbol{k}$ is determined up to isomorphism by its comodule category $\mathbb{M}^A$ whose objects are of finite dimension \cite{Deligne, Majid, Schauen_Tannakian, Ulbrich}. 
More precisely, let $A$ and $B$ be two finite-dimensional bialgebras over $\boldsymbol{k}$, and 
$F : \mathbb{M}^A \longrightarrow \mathbb{M}^B$ be a $\boldsymbol{k}$-linear monoidal functor. 
If $F$ is fibered, then 
there is a bialgebra map $\varphi : A\longrightarrow B$ such that $F=\mathbb{M}^{\varphi}$, where $\mathbb{M}^{\varphi}$ is the induced monoidal functor from $f$. 
This statement can be found in Majid's book \cite[Theorem 2.2]{Majid_book} and a detailed proof is given in Franco's lecture note \cite[p.80--84]{Franco}. 
\par 
In this note we treat the weak bialgebra version of the above result. It is noted that 
as is well known, the reconstruction theorem for weak bialgebras has been established by several researchers \cite{Chikhladze, H-Oldenburg, McCrudy, BV} (see \cite{Vercruysse} for other variants). 
However, it seems that there is no statement on a reconstruction theorem for ``weak bialgebra map" in any papers 
though it is very fundamental. 
This note is devoted to give a precise statement of it and prove it. 
In fact, although the same statement with the classical one holds, the proof is a rather complicated since 
the unit object in the comodule category over a weak bialgebra $A$ is not necessary to the base field $\boldsymbol{k}$. 
Actually, it is a subalgebra, which is called the source counital subalgebra of $A$. 
For that reason, the proof of the reconstruction theorem for weak bialgebra maps is accomplished by re-examining the method used in the proof of the reconstruction of coalgebra map.  
\par 
Recently, the author study on indecomposability of weak bialgebras, and find some interesting results and unsolved problems \cite{Wakui_ContempMath}. 
As an application of the reconstruction theorem of weak bialgebra maps, 
we derive a categorical interpretation of indecomposability of a weak bialgebra. 
\par 
This paper is organized as follows. 
In Section 2 we recall the definition and basic properties of weak bialgebras, and also the comodule structure over them. 
In Section 3 after we overview the proof of the reconstruction theorem of coalgebra maps, we state and prove the reconstruction theorem of bialgebra maps. 
In Section 4, that is the final section, we apply the theorem to characterize indecomposability of finite-dimensional weak bialgebras. 
\par 
Throughout this note, $\boldsymbol{k}$ denotes a field, and 
$\text{Vect}_{\boldsymbol{k}}^{\text{f.d.}}$ stands for the monoidal category whose objects are finite-dimensional vector spaces over $\boldsymbol{k}$ and morphisms are $\boldsymbol{k}$-linear maps between them. 
For a weak bialgebra or a weak Hopf algebra $H$, we denote by $\Delta _H$, $\varepsilon _H$ and $S_H$ the comultiplication, the counit and the antipode of $H$, respectively. 
When it is clear that they are for $H$, they are simply denoted by $\Delta $, $\varepsilon $ and $S$, respectively. 
The notation $\Delta ^{(2)}$ means the composition $(\Delta \otimes \text{id})\circ \Delta =(\text{id}\otimes \Delta )\circ \Delta $. 
We use Sweedler's notation such as $\Delta (x)=x_{(1)}\otimes x_{(2)}$ for $x\in H$, and 
for a right $H$-comodule $M$ with coaction $\rho $ we also use the notation $\rho (m)=m_{(0)}\otimes m_{(1)}$ for $m\in M$. 
\par 
A monoidal category is described as $\mathcal{C}=(\mathcal{C}, \otimes , I, a, l, r)$, as in MacLane's book \cite{MacLane_book}, 
and a monoidal functor between monoidal categories $\mathcal{C}$ and $\mathcal{D}$ is described as $F=(F, \Phi ^F, \omega ^F)$, 
where $F$ is a covariant functor from $\mathcal{C}$ to $\mathcal{D}$, and $\Phi ^F$ is a natural transformation obtained by collecting morphisms $\phi^F_{M, N}: F(M)\otimes F(N) \longrightarrow F(M\otimes N)$ for all $M, N\in \mathcal{C}$, 
and $\omega ^F : F(I_{\mathcal{C}})\longrightarrow I_{\mathcal{D}}$ is a morphism, 
where $I_{\mathcal{C}}$ and $I_{\mathcal{D}}$ are the unit objects in $\mathcal{C}$, and $\mathcal{D}$, respectively. 
If $\Phi^F$ is a natural equivalence and $\omega ^F$ is an isomorphism, then the monoidal functor $(F, \Phi ^F, \omega ^F)$ is called strong. 
As in the classical case, every weak bialgebra map $\varphi : H\longrightarrow K$ induces a covariant functor $\mathbb{M}^{\varphi} : \mathbb{M}^H \longrightarrow \mathbb{M}^K$. 
We note that the functor $\mathbb{M}^{\varphi}$ is not monoidal but is comonoidal. 
By a comonoidal functor we mean a triplet $F=(F, \bar{\Phi }^F, \bar{\omega} ^F)$, which all arrows in $\bar{\Phi }^F$ and $\bar{\omega} ^F$ are in reverse in monoidal categories, 
that is, $\bar{\Phi }^F$ is a natural transformation obtained by collecting morphisms $\bar{\phi}^F_{M, N}: F(M\otimes N) \longrightarrow F(M)\otimes F(N)$ for all $M, N\in \mathcal{C}$, and $\bar{\omega} ^F : I_{\mathcal{D}}\longrightarrow F(I_{\mathcal{C}})$ is a morphism. 
By the same condition for a monoidal functor, the concept called strong is defined for a comonoidal functor. 
In some literatures terminologies ``colax" or ``op-monoidal" are used for ``comonoidal". 
\par 
For general facts on Hopf algebras, we refer the reader to Montgomery's book \cite{Montgomery_book}. 

\par \medskip 
\section{Definition of weak bialgebras and structures of their comodule categories}
In this section we recall the definition and basic properties of weak bialgebras, and also the comodule structure over them mainly following \cite{BNS} and \cite{NTV}. 

\par 
Let $H$ be a vector space over $\boldsymbol{k}$, and $(H, \mu , \eta )$ is an algebra and 
$(H, \Delta , \varepsilon )$ is a coalgebra over $\boldsymbol{k}$. 
The $5$-tuple $(H, \mu , \eta , \Delta , \varepsilon )$ is said to be a \textit{weak bialgebra} over $\boldsymbol{k}$ if the following three conditions are satisfied. 
\begin{enumerate}\itemindent=1cm 
\item[(WH1)] $\Delta (xy)=\Delta (x)\Delta (y)$ for all $x, y\in H$.
\item[(WH2)] $\Delta ^{(2)}(1)=(\Delta (1)\otimes 1)(1\otimes \Delta (1))=
(1\otimes \Delta (1))(\Delta (1)\otimes 1)$, 
where $1=\eta (1)$ is the identity element of the algebra $(H, \mu , \eta )$. 
\item[(WH3)] For all $x,y,z\in H$
\begin{enumerate}
\item[(i)] $\varepsilon (xyz)=\varepsilon (xy_{(1)})\varepsilon (y_{(2)}z)$.
\item[(ii)] $\varepsilon (xyz)=\varepsilon (xy_{(2)})\varepsilon (y_{(1)}z)$. 
\end{enumerate}
\end{enumerate}

Let $S: H\longrightarrow H$ be a $\boldsymbol{k}$-linear map. 
The $6$-tuple $(H, \mu , \eta , \Delta , \varepsilon ,S)$ is said to be a \textit{weak Hopf algebra} over $\boldsymbol{k}$ if the above three conditions and the following additional condition are satisfied.
\begin{enumerate}\itemindent=1cm 
\item[(WH4)] For all $x\in H$
\begin{enumerate}
\item[(i)] $x_{(1)}S(x_{(2)})=\varepsilon (1_{(1)}x)1_{(2)}$. 
\item[(ii)] $S(x_{(1)})x_{(2)}=1_{(1)}\varepsilon (x1_{(2)})$.
\item[(iii)] $S(x_{(1)})x_{(2)}S(x_{(3)})=S(x)$. 
\end{enumerate}
\end{enumerate}

The above $S$ is called the antipode of $(H, \mu , \eta , \Delta , \varepsilon )$ and $(H, \mu , \eta , \Delta , \varepsilon ,S)$. 
We note that it is unique if it exists. 

\par 
For a weak bialgebra $H=(H, \mu , \eta , \Delta , \varepsilon )$, by the condition (WH4)(i),(ii), 
the following two $\boldsymbol{k}$-linear maps $\varepsilon _t, \varepsilon _s: H\longrightarrow H$ are defined: 
\begin{align}
\varepsilon_t(x)=\varepsilon (1_{(1)}x)1_{(2)},  \label{eq_target_counital_map} \\ 
\varepsilon_s(x)=1_{(1)}\varepsilon (x1_{(2)}). \label{eq_source_counital_map}
\end{align}

The maps $\varepsilon _t$ and $\varepsilon _s$ are called the \textit{target counital map} and the \textit{source counital map}, respectively. 
These maps satisfy the following properties. 

\par \medskip 
\begin{lem}\label{2-1}
\begin{enumerate}
\item[$(1)$] $\varepsilon _t^2=\varepsilon _t,\ \varepsilon _s^2=\varepsilon _s$.
\item[$(2)$] For all $x\in H$
\begin{enumerate}
\item[$(i)$] $((\mathrm{id}\otimes \varepsilon _t)\circ \Delta )(x)=1_{(1)}x\otimes 1_{(2)}$. 
\item[$(ii)$] $((\varepsilon _s\otimes \mathrm{id})\circ \Delta )(x)=1_{(1)}\otimes x1_{(2)}$. 
\end{enumerate}
In particular, 
\begin{equation}\label{eq2-3}
1_{(1)}\otimes \varepsilon _t(1_{(2)})=1_{(1)}\otimes 1_{(2)}=\varepsilon _s(1_{(1)})\otimes 1_{(2)}.
\end{equation}
\item[$(3)$] For all $x\in H$
\begin{enumerate}
\item[$(i)$] $\varepsilon _t(x)=x\ \ \Longleftrightarrow \ \ \Delta (x)=1_{(1)}x\otimes 1_{(2)}$. 
\item[$(ii)$] $\varepsilon _s(x)=x\ \ \Longleftrightarrow \ \ \Delta (x)=1_{(1)}\otimes x1_{(2)}$. 
\end{enumerate}
\end{enumerate}
Especially, by Part $(2)$$(i)$
\begin{align*}
1_{(1)}1_{[1]}\otimes 1_{(2)}\otimes 1_{[2]}&=1_{(1)}\otimes \varepsilon_t(1_{(2)})\otimes 1_{(3)},\\ 
1_{(1)}\otimes 1_{[1]}\otimes 1_{(2)}1_{[2]}&=1_{(1)}\otimes \varepsilon_s(1_{(2)})\otimes 1_{(3)}, 
\end{align*}
where 
$\Delta ^{(2)}(1)=1_{(1)}\otimes 1_{(2)}\otimes 1_{(3)}$ and $\Delta (1)=1_{[1]}\otimes 1_{[2]}$. 
\end{lem} 

\newpage 
\par \medskip 
\begin{lem}\label{2-2}
Let $H$ be a weak bialgebra over $\boldsymbol{k}$. For all $x,y\in H$
\begin{enumerate}
\item[$(1)$] $\varepsilon_t(x\varepsilon_t(y))  =\varepsilon_t(xy) $, 
$\varepsilon_s(\varepsilon_s(x)y)  =\varepsilon_s(xy)$.
\item[$(2)$] 
$\varepsilon (x\varepsilon_t(y)) =\varepsilon (xy)$,
$\varepsilon (\varepsilon_s(x)y) =\varepsilon (xy)$.
\item[$(3)$] $\varepsilon \circ \varepsilon _t=\varepsilon =\varepsilon \circ \varepsilon _s$.
\item[$(4)$] $x=\varepsilon_t(x_{(1)})x_{(2)} =x_{(1)}\varepsilon_s(x_{(2)})$. 
\item[$(5)$] 
$x\varepsilon_t(y)  =\varepsilon_t(x_{(1)}y)x_{(2)}$,
$\varepsilon_s(x)y  =y_{(1)}\varepsilon_s(xy_{(2)})$.
\end{enumerate}
\end{lem}

\par \medskip 
\begin{lem}\label{2-3}
Let $H$ be a weak bialgebra over $\boldsymbol{k}$, and set 
$H_t:=\varepsilon _t(H),\ H_s:=\varepsilon _s(H)$. Then
\begin{enumerate}
\item[$(1)$] $x\varepsilon _t(y)=\varepsilon _t(xy)$ for all $x\in H_t$ and $y\in H$. 
\item[$(2)$] $\varepsilon _s(x)y=\varepsilon _s(xy)$ for all $x\in H$ and $y\in H_s$. 
\item[$(3)$] 
\begin{enumerate}
\item[$(i)$] An element of $H_t$ and an element of $H_s$ commute, and 
\item[$(ii)$] $H_t$ and $H_s$ are a left coideal and a right coideal subalgebras of $H$, respectively. 
\end{enumerate}
\end{enumerate}
The subalgebras $H_t$ and $H_s$ are called the target and source subalgebras of $H$, respectively. 
By \eqref{eq2-3} we have 
\begin{equation}\label{eq2-4}
\Delta (1)\in H_s\otimes H_t.
\end{equation}
\begin{enumerate}
\item[$(4)$] For all $x\in H$, $z\in H_t$ and $y\in H_s$
\begin{align*}
\Delta (xz)=x_{(1)}z\otimes x_{(2)},\quad & 
\Delta (zx)=zx_{(1)}\otimes x_{(2)},\\ 
\Delta (xy)=x_{(1)}\otimes x_{(2)}y,\quad & 
\Delta (yx)=x_{(1)}\otimes yx_{(2)}.
\end{align*}
In particular 
\begin{align*}
xz=\varepsilon (x_{(1)}z)x_{(2)},\quad & 
zx=\varepsilon (zx_{(1)})x_{(2)},\\ 
xy=x_{(1)}\varepsilon (x_{(2)}y),\quad & 
yx=x_{(1)}\varepsilon (yx_{(2)}). 
\end{align*}
\end{enumerate}
\end{lem}

\par \medskip 
For a weak bialgebra $H$, one can also consider two 
$\boldsymbol{k}$-linear maps $\varepsilon _t^{\prime}, \varepsilon _s^{\prime}: H\longrightarrow H$ defined by 
\begin{align}
\varepsilon _t^{\prime}(x)&=\varepsilon (x1_{(1)})1_{(2)}, \label{eq2-5} \\ 
\varepsilon _s^{\prime}(x)&=1_{(1)}\varepsilon (1_{(2)}x) \label{eq2-6} 
\end{align}
\noindent 
for all $x\in H$. 
Then, $H^{\mathrm{op}}=(H, \mu^{\mathrm{op}}, \eta , \Delta , \varepsilon )$, $H^{\mathrm{cop}}=(H, \mu , \eta , \Delta ^{\mathrm{cop}}, \varepsilon )$, $H^{\mathrm{op cop}}=(H, \mu^{\mathrm{op}}, \eta , \Delta ^{\mathrm{cop}}$, $\varepsilon )$ are weak bialgebras over $\boldsymbol{k}$, 
where $\mu^{\mathrm{op}}$ and $\Delta ^{\mathrm{cop}}$ are the opposite multiplication and comultiplication, respectively. 
The target and the source subalgebras of them are given by
$(H^{\mathrm{op}})_t=H_t,\ (H^{\mathrm{op}})_s=H_s$,\ 
$(H^{\mathrm{cop}})_t=H_s,\ (H^{\mathrm{cop}})_s=H_t$,\ 
$(H^{\mathrm{op cop}})_t=H_s,\ 
(H^{\mathrm{op cop}})_s=H_t$. 
The target and the source counital maps of them are given by 
$(\varepsilon _{H^{\mathrm{op}}})_t=\varepsilon _t^{\prime},\ (\varepsilon _{H^{\mathrm{op}}})_s=\varepsilon _s^{\prime}$, 
$(\varepsilon _{H^{\mathrm{cop}}})_t=\varepsilon _s^{\prime},\ (\varepsilon _{H^{\mathrm{cop}}})_s=\varepsilon _t^{\prime}$,\ 
$(\varepsilon _{H^{\mathrm{op cop}}})_t=\varepsilon _s,\ (\varepsilon _{H^{\mathrm{op cop}}})_s=\varepsilon _t$. 
If $S$ is the antipode of $H$, then it is also of $H^{\mathrm{op cop}}$. 

\par 
For a weak bialgebra $H$ over $\boldsymbol{k}$, we denote by $\Mod^H$ the $\boldsymbol{k}$-linear category whose objects are right $H$-comodules and morphisms are $H$-comodule maps between them. 
The comodule category $\Mod^H$ has a monoidal structure~\cite[Section 4]{NTV} as the following lemma. 

\par \medskip 
\begin{lem}[{\cite[Lemma 4.2]{Nill}}]\label{2-4}
Let $H$ be a weak bialgebra over $\boldsymbol{k}$. 
\begin{enumerate}
\item[$(1)$] For two right $H$-comodules $(M,\rho _M),\ (N, \rho _N)$, the pair 
$(M,\rho _M)\circledast (N, \rho _N):=(M\otimes _{H_s}N, \rho )$ is also a right $H$-comodule, 
where $\rho : M\otimes _{H_s}N\longrightarrow (M\otimes _{H_s}N)\otimes H$ is a $\boldsymbol{k}$-linear map defined by 
\begin{equation}
\rho (m\otimes _{H_s}n)=(m_{(0)}\otimes _{H_s}n_{(0)})\otimes m_{(1)}n_{(1)}\qquad (m\in M,\ n\in N). 
\end{equation}

The source algebra $H_s$ can be regarded as a right $H$-comodule with the coaction $\Delta _s:=\Delta |_{H_s}: H_s: \longrightarrow H_s\otimes H$, and for all right $H$-comodules $L, M, N$ there are natural isomorphisms 
\begin{enumerate}
\item[$(i)$] $H_s\circledast M\cong M\cong M\circledast H_s$ as right $H$-comodules,
\item[$(ii)$] $(L\circledast M)\circledast N\cong L\circledast (M\circledast N)$ as right $H$-comodules. 
\end{enumerate}
Here the natural isomorphisms in $(i)$ are given as follows: 
For a right $H$-comodule $M$
\begin{align*}
l_M : &\ H_s\otimes _{H_s}M\longrightarrow M,\  \quad l_M(y\otimes _{H_s}m)=y\cdot m\qquad (y\in H_s,\ m\in M),\\ 
l_M^{-1}: &\ M\longrightarrow H_s\otimes _{H_s}M ,\  \quad l_M^{-1}(m)=1\otimes _{H_s}m\qquad \quad (m\in M),\\ 
r_M : &\ M\otimes _{H_s}H_s\longrightarrow M,\ \quad r_M(m\otimes _{H_s}y)=m\cdot y\qquad (m\in M,\ y\in H_s),\\ 
r_M^{-1}: &\  M\longrightarrow M\otimes _{H_s}H_s,\  \quad r_M^{-1}(m)=m\otimes _{H_s}1\qquad \quad (m\in M). 
\end{align*}

The isomorphism in $(ii)$ is induced from a usual isomorphism between vector spaces. 
\item[$(2)$] Let $f: (M,\rho _M)\longrightarrow (N, \rho _N)$, $g: (M^{\prime},\rho _M^{\prime})\longrightarrow (N^{\prime}, \rho _N^{\prime})$ be right $H$-comodule maps. Then
$f\otimes _{H_s}g: M\otimes _{H_s}M^{\prime} \longrightarrow N\otimes _{H_s}N^{\prime}$ is also a right $H$-comodule map with respect to the comodule structure given in $(1)$. 
We denote the map $f\otimes _{H_s}g$ by $f\circledast g$. 
\end{enumerate}
By Parts $(1)$, $(2)$ the abelian category $\Mod ^H$ becomes a $\boldsymbol{k}$-linear monoidal category whose unit object is $(H_s, \Delta _s)$. 
\end{lem}

\par \medskip 
\begin{lem}[{\cite[Proposition 4.1]{Nill}}]\label{2-5}
Let $H$ be a weak bialgebra over $\boldsymbol{k}$, and $(M, \rho _M)$ be a right $H$-comodule. 
For $y\in H_s$ and $m\in M$, the elements $y\cdot m$ and $m\cdot y$ in $M$ are defined as follows: 
\begin{align}
y\cdot m &:=m_{(0)}\varepsilon (ym_{(1)}), \\ 
m\cdot y &:=m_{(0)}\varepsilon (m_{(1)}y). 
\end{align}
\begin{enumerate}
\item[$(1)$] $M$ becomes an $(H_s, H_s)$-bimodule equipped with the above actions. 
\item[$(2)$] $M\otimes H$ becomes an $(H_s, H_s)$-bimodule equipped with the following actions: 
For $y\in H_s,\ m\in M,\ x\in H$, 
\begin{align}
y\cdot (m\otimes x)&:=(1_{(1)}\cdot m)\otimes (y1_{(2)}x)=m_{(0)}\otimes \varepsilon_t(m_{(1)})yx, \\ 
(m\otimes x)\cdot y &:=(m\cdot 1_{(1)})\otimes (xy1_{(2)})=m_{(0)}\otimes xy\varepsilon_t^{\prime}(m_{(1)}). 
\end{align}
\item[$(3)$] The following equations hold for all $y\in H_s$ and $m\in M$: 
\begin{enumerate}
\item[$(i)$] $\rho _M(y\cdot m)=m_{(0)}\otimes ym_{(1)}=y\cdot \rho _M(m)$. 
\item[$(ii)$] $\rho _M(m\cdot y)=m_{(0)}\otimes m_{(1)}y=\rho _M(m)\cdot y$. 
\item[$(iii)$] $\varepsilon _s^{\prime}(m_{(1)})\cdot m_{(0)}=m=m_{(0)}\cdot \varepsilon _s(m_{(1)})$. 
\end{enumerate}
In particular, $\rho_M: M \longrightarrow M\otimes H$ is an $(H_s, H_s)$-bimodule map by $(i)$ and $(ii)$.  
\item[$(4)$] Let $(N, \rho _N)$ be a right $H$-comodule and $f: (M, \rho _M)\longrightarrow (N, \rho _N)$ be an $H$-comodule map. 
Then $f$ becomes an $(H_s, H_s)$-bimodule map with respect to the bimodule structures given by $(1)$. 
\end{enumerate}
\end{lem}

\par \medskip 
Let $H$ be a weak bialgebra over $\boldsymbol{k}$, and denote by ${}_{H_s}\Mod_{H_s}$ the $\boldsymbol{k}$-linear category consisting of whose objects are $(H_s, H_s)$-bimodules and morphisms are $(H_s, H_s)$-bimodule maps between them. 
By Lemma~\ref{2-5} for a right $H$-comodule $(M, \rho_M)$, 
 the underlying vector space $M$ has an $(H_s, H_s)$-bimodule structure, and $\rho_M: M \longrightarrow M\otimes H$ becomes an $(H_s, H_s)$-bimodule map. 
Thus $(M, \rho_M)$ can be regarded as a right $H$-comodule in ${}_{H_s}\Mod_{H_s}$. 
Then any $H$-comodule map  $f: M\longrightarrow N$ always an $(H_s, H_s)$-bimodule map. 
\par 
We denote by ${}_{H_s}\Mod_{H_s}^H$ the $\boldsymbol{k}$-linear category whose objects are right $H$-comodules in ${}_{H_s}\Mod_{H_s}$ and morphisms are right $H$-comodule and $(H_s, H_s)$-bimodule maps between them. 
The category ${}_{H_s}\Mod_{H_s}^H$ has a $\boldsymbol{k}$-linear monoidal structure whose tensor product is given by 
$\otimes_{H_s}$. 
\par 
As a special case of \cite[Theorem 2.2]{Szlach05} we have: 

\par \medskip 
\begin{lem}\label{2-6}
Let $H$ be a weak bialgebra over $\boldsymbol{k}$. 
By Lemma~\ref{2-5} for a right $H$-comodule $(M, \rho_M)$, 
there is an $(H_s, H_s)$-bimodule structure on the underlying vector space $M$, 
and $\rho_M: M \longrightarrow M\otimes H$ becomes an $(H_s, H_s)$-bimodule map. 
Thus $(M, \rho_M)$ can be regarded as a right $H$-comodule in ${}_{H_s}\Mod_{H_s}$. 
Then any $H$-comodule map  $f: M\longrightarrow N$ is always an $(H_s, H_s)$-bimodule map. 
This correspondence gives rise to a $\boldsymbol{k}$-linear monoidal equivalence $\Xi ^H: \Mod^H \longrightarrow {}_{H_s}\Mod_{H_s}^H$ between $\boldsymbol{k}$-linear monoidal categories. 
\end{lem}

\par \medskip 
We set 
$$\hat{U}^H:=U^H\circ \Xi ^H: \Mod^H \longrightarrow {}_{H_s}\Mod_{H_s},$$
where $U^H: {}_{H_s}\Mod_{H_s}^H \longrightarrow {}_{H_s}\Mod_{H_s}$ is the forgetful monoidal functor. 
The composition $\hat{U}^H$ of monoidal functors becomes a monoidal functor, whose structure is given as follows: 
\begin{enumerate}
\item[$\bullet$] $\phi _{M, N}^{\hat{U}^H}=\text{id}_{M\otimes _{H_s}N}$ for each $M, N\in  \Mod^H$, 
\item[$\bullet$] $\omega ^{\hat{U}^H}: H_s \longrightarrow \hat{U}^H(H_s)=H_s$ is the identity map. 
\end{enumerate}

\par \medskip 
\begin{lem}\label{2-7}
Let $H, K$ be weak bialgebras over $\boldsymbol{k}$, and $\varphi : H\longrightarrow K$ be a weak bialgebra map, namely, 
an algebra map and coalgebra map. 
Then $\varphi (H_s)\subset K_s$ and $\varphi (H_t)\subset K_t$. 
The former inclusion induces an algebra map 
$\varphi _s:=\varphi |_{H_s}: H_s \longrightarrow K_s$. 
\begin{enumerate}
\item[$(1)$] For a right $H$-comodule $(M, \rho_M)$
$$\Mod^{\varphi }(M, \rho_M):=(M, (\text{id}_M\otimes \varphi )\circ \rho_M)$$ 
is a right $K$-comodule, and for a right $H$-comodule map 
$f : (M, \rho_M) \longrightarrow (N, \rho_N)$ 
$$\Mod^{\varphi }(f):=f: \Mod^{\varphi }(M, \rho_M)\longrightarrow \Mod^{\varphi }(N, \rho_N)$$
is a right $K$-comodule map. 
These correspondences define a covariant functor 
$\Mod^{\varphi}: \Mod^H \longrightarrow \Mod^K$. 
\item[$(2)$] The functor $\Mod^{\varphi}$ is a $\boldsymbol{k}$-linear comonoidal functor. 
If $\varphi _s$ is bijective, then the $\boldsymbol{k}$-linear comonoidal functor $\Mod^{\varphi}$ is strong. 
In this case it can be regarded as a $\boldsymbol{k}$-linear monoidal functor. 
\item[$(3)$] The algebra map $\varphi _s$ induces a $\boldsymbol{k}$-linear monoidal functor ${}_{\varphi _s}\Mod_{\varphi _s}: {}_{K_s}\Mod_{K_s}\longrightarrow {}_{H_s}\Mod_{H_s}$, and if $\varphi _s$ is bijective, then $\hat{U}^K\circ \Mod^{\varphi}={}_{\varphi _s^{-1}}\Mod_{\varphi _s^{-1}}\circ \hat{U}^H : \Mod^H \longrightarrow {}_{K_s}\Mod_{K_s}$ as monoidal functors. 
\end{enumerate}
\end{lem}
\begin{proof} 
(2) For right $H$-comodules $(M, \rho_M), (N, \rho_N)$ we set 
\begin{align*}
\Mod^{\varphi }(M, \rho_M)\circledast \Mod^{\varphi }(N, \rho_N)&=(M\otimes _{K_s} N, \rho ),\\ 
\Mod^{\varphi }\bigl( (M, \rho_M)\circledast (N, \rho_N)\bigr) &=(M\otimes _{H_s} N, \rho ^{\prime}), 
\end{align*}
where $\rho$ and $\rho^{\prime}$ are given as follows: 
\begin{align*}
\rho : M\otimes _{K_s} N\longrightarrow (M\otimes _{K_s} N)\otimes K,\quad & 
\rho (m\otimes _{K_s} n)=m_{(0)}\otimes _{K_s} n_{(0)} \otimes \varphi (m_{(1)})\varphi (n_{(1)}),\\ 
\rho ^{\prime}: M\otimes _{H_s} N\longrightarrow (M\otimes _{H_s} N)\otimes K,\quad &
\rho ^{\prime}(m\otimes _{H_s} n)=m_{(0)}\otimes _{H_s} n_{(0)} \otimes \varphi (m_{(1)}n_{(1)}). 
\end{align*}

Since $\varphi$ is an algebra map, the identity map $\text{id}_{M\otimes N}$ induces a surjection $\iota _{M, N}: M\otimes _{H_s} N \longrightarrow M\otimes_{K_s}N$ which is a right $K$-comodule map. 
\par 
The restriction $\varphi _s: H_s \longrightarrow K_s$ can be regarded as a right $K$-comodule map from \newline $\Mod^{\varphi }(H_s, (\Delta _H)_s)=(H_s,\ (\text{id}_{H_s}\otimes \varphi )\circ (\Delta _H)_s)$ to $(K_s, (\Delta _K)_s)$. 
Thus, the triplet $(\Mod^{\varphi}, \iota , \varphi_s)$ is a comonoidal functor from $\Mod^H$ to $\Mod^K$, where  $\iota :=\{ \iota_{M, N}\}_{M,N \in \Mod^H}$. 
\par 
If $\varphi_s$ is bijective, then $\iota _{M, N}$ for all $M, N\in \Mod^H$ is an isomorphism. 
Thus, $(\Mod^{\varphi}, \iota^{-1} , \varphi_s^{-1}): \Mod^H \longrightarrow \Mod^K$ is a strong monoidal functor. 
\par 
(3) For a $(K_s, K_s)$-bimodule $(M, \alpha _l, \alpha _r)$ we set 
$${}_{\varphi_s}\Mod_{\varphi_s}(M, \alpha _l, \alpha _r)=\bigl( M, \alpha _l\circ (\varphi_s\otimes \text{id}_M), \alpha_r\circ (\text{id}_M\otimes \varphi_s)\bigr) ,$$
and a $(K_s, K_s)$-bimodule map $f: M \longrightarrow N$ we set 
$${}_{\varphi_s}\Mod_{\varphi_s}(f)=f.$$
Then, ${}_{\varphi_s}\Mod_{\varphi_s}$ is a $\boldsymbol{k}$-linear covariant functor. 
\par 
For $(K_s, K_s)$-bimodules $M=(M, \alpha _l^M, \alpha _r^M)$ and $N=(N, \alpha _l^N, \alpha _r^N)$, let 
$$\jmath _{M, N}: {}_{\varphi_s}\Mod_{\varphi_s}(M)\otimes _{H_s}{}_{\varphi_s}\Mod_{\varphi_s}(N) \longrightarrow {}_{\varphi_s}\Mod_{\varphi_s}(M\otimes _{K_s} N)$$
be the induced $\boldsymbol{k}$-linear map from the identity map $\text{id}_{M_\otimes N}$. 
This is an $(H_s, H_s)$-module map which is natural with respect to $M ,N$. 
So, 
we have a monoidal functor $({}_{\varphi_s}\Mod_{\varphi_s}, \jmath , \varphi_s)$ by setting $\jmath=\{ \jmath_{M, N}\}_{M, N\in {}_{K_s}\Mod_{K_s}}$. 
If $\varphi_s$ is bijective, then 
one can show that $\hat{U}^K\circ \Mod^{\varphi}={}_{\varphi _s^{-1}}\Mod_{\varphi _s^{-1}}\circ \hat{U}^H$ as $\boldsymbol{k}$-linear monoidal functors. 
\end{proof} 

\par \bigskip 
\section{Reconstruction of a weak bialgebra map}
\par 
First of all, we will recall the proof of the reconstruction theorem of coalgebra maps, that is a classical and fundamental theorem for Tannakian reconstruction theorem. 

\par 
For a coalgebra $C$ over $\boldsymbol{k}$ we denote by $\mathbb{M}^C$ 
the $\boldsymbol{k}$-linear category consisting of whose objects are right $C$-comodules of finite dimension and morphisms are $C$-comodule maps between them, and denote by $U^C$ 
the forgetful functor from $\mathbb{M}^C$ to $\mathrm{Vect}_{\boldsymbol{k}}^{\mathrm{f.d.}}$. 

\par \medskip 
\begin{thm}\label{3-1}
Let $C, D$ be coalgebras over $\boldsymbol{k}$, and $F: \mathbb{M}^C \longrightarrow \mathbb{M}^D$ be a $\boldsymbol{k}$-linear covariant functor. 
If $U^D\circ F=U^C$, then there is a unique coalgebra map $\varphi : C\longrightarrow D$ such that $F=\mathbb{M}^{\varphi }: \mathbb{M}^C \longrightarrow \mathbb{M}^D$. 
Here, $\mathbb{M}^{\varphi }$ denotes the $\boldsymbol{k}$-linear functor induced from $\varphi $. 
\end{thm} 
\begin{proof} 
The proof follows from Franco's lecture note~\cite[p.81--84]{Franco}. 
\par 
Let $(M, \rho_M)$ be a finite-dimensional $C$-comodule. 
By the assumption $U^D\circ F=U^C$, 
one can set $F(M, \rho_M)=(M, \rho_M^F)$, where 
$\rho_M^F: M\longrightarrow M\otimes D$ is a right coaction of $D$. 
\par 
Let $P$ be a finite-dimensional subcoalgebra of $C$. 
It can be regarded as a right $C$-comodule by the coaction 
$$\rho_P: P \xrightarrow{\ \ \Delta _P\ \ } P\otimes P \xrightarrow{\  \ \text{id}\otimes \iota _P\ \ } P\otimes C,$$
where $\iota_P$ stands for the inclusion. 
Since $P$ is finite-dimensional, it gives an object of $\mathbb{M}^C$, and therefore we have 
$F(P, \rho_P)=(P, \rho_P^F)\in \mathbb{M}^D$. 
\par 
Let us consider the composition 
$$\varphi _P : P \xrightarrow{\ \ \rho_P^F\ \ } P\otimes D \xrightarrow{\ \varepsilon _P\otimes\text{id}\ } \boldsymbol{k}\otimes D \cong D.$$
Then 
$\varphi _P: P\longrightarrow D$ is a coalgebra map. 
This fact can be verified as follows. 

\par \noindent 
$\bullet$ The equation $\varepsilon _D\circ \varphi _P=\varepsilon _P$ comes from the following commutative diagram. 

{\setlength{\unitlength}{0.7cm}
\begin{center}
\begin{picture}(17,4.5)
\put(1,3.5){$P$}
\put(5,3.5){$P\otimes D$}
\put(10,3.5){$\boldsymbol{k}\otimes D$}
\put(11.9, 3.5){$\cong$}
\put(13,3.5){$D$}
\put(1,0.5){$P$}
\put(3, 0.5){$\cong$}
\put(5, 0.5){$P\otimes \boldsymbol{k}$}
\put(10, 0.5){$\boldsymbol{k}\otimes \boldsymbol{k}$}
\put(11.9, 0.5){$\cong$}
\put(13, 0.5){$\boldsymbol{k}$}
\put(3.1, 3.9){\small $\rho_P^F$}
\put(7.5, 3.9){\small $\varepsilon_P\otimes \text{id}$}
\put(7.5, 0.9){\small $\varepsilon_P\otimes \text{id}$}
\put(0.7, 2.2){\small $\text{id}$}
\put(6, 2.2){\small $\text{id}\otimes \varepsilon _D$}
\put(10.9, 2.2){\small $\text{id}\otimes \varepsilon _D$}
\put(13.4, 2.2){\small $\varepsilon _D$}
\put(1.7,3.6){\vector(1,0){3}}
\put(6.8,3.6){\vector(1,0){3}}
\put(6.8,0.6){\vector(1,0){3}}
\put(1.2, 3.2){\vector(0,-1){2}}
\put(5.8, 3.2){\vector(0,-1){2}}
\put(10.7, 3.2){\vector(0,-1){2}}
\put(13.2, 3.2){\vector(0,-1){2}}
\end{picture}
\end{center}
}


\noindent 
$\bullet$ To show the equation $\Delta _D\circ \varphi _P=(\varphi _P\otimes \varphi _P)\circ \Delta _P$, 
it is enough to verify the following diagram commutes: 

{\setlength{\unitlength}{0.3mm}
\begin{center}
\begin{picture}(430,225)(60,-100)
\put(50,100){$P$}
\put(150,100){$P\otimes P$}
\put(400,100){$P\otimes D\otimes P\otimes D$}
\put(100, 50){$P\otimes D$}
\put(200, 50){$P\otimes P\otimes D$}
\put(150, 0){$P\otimes D\otimes D$}
\put(225, 0){$\cong$}
\put(250, 0){$P\otimes D\otimes \boldsymbol{k}\otimes D$}
\put(400, 0){$\boldsymbol{k}\otimes D\otimes \boldsymbol{k}\otimes D$}
\put(300, -50){$\boldsymbol{k}\otimes D\otimes D$}
\put(50, -100){$P\otimes D$}
\put(200, -100){$\boldsymbol{k}\otimes D\cong D$}
\put(420, -100){$D\otimes D$}
\put(80,0){(A)}
\put(135,75){(B)}
\put(220,25){(C)}
\put(100,110){\small $\Delta_P$}
\put(280,110){\small $\rho_P^F\otimes \rho_P^F$}
\put(80,83){\small $\text{id}\otimes \rho_P^F$}
\put(210,83){\small $\rho_P^F$}
\put(148, 60){\small $\Delta_P\otimes \text{id}$}
\put(132, 32){\small $\rho_P^F\otimes \text{id}$}
\put(120, -50){\small $\text{id}\otimes \Delta _D$}
\put(40, 2){\small $\rho_P^F$}
\put(122, -93){\small $\varepsilon_P\otimes \text{id}$}
\put(327, -93){\small $\Delta_P$}
\put(315, 83){\small $\rho_P^F\otimes \text{id}$}
\put(340, 43){\small $\text{id}\otimes \varepsilon _P\otimes \text{id}$}
\put(260, -23){\small $\varepsilon _P\otimes \text{id}$}
\put(370, -70){\small $\cong$}
\put(445, 50){\small $\varepsilon _P\otimes \text{id}\otimes \varepsilon _P\otimes \text{id}$}
\put(445, -50){\small $\cong $}
\put(63.5,104){\vector(1,0){80}}
\put(194,104){\vector(1,0){200}}
\put(63.5,94){\vector(1,-1){33}}
\put(194,94){\vector(1,-1){33}}
\put(143.5,54){\vector(1,0){50}}
\put(113.5, 44){\vector(1,-1){33}}
\put(57.5, 94){\vector(0,-1){175}}
\put(73.5, -84){\vector(1,1){75}}
\put(93, -96){\vector(1,0){95}}
\put(267, -96){\vector(1,0){140}}
\put(183.5, -10){\vector(4,-1){110}}
\put(343.5, -58){\vector(2,-1){65}}
\put(440, 93){\vector(0,-1){75}}
\put(440, -11){\vector(0,-1){75}}
\put(270, 61){\vector(4,1){125}}
\put(410, 91){\vector(-1,-1){75}}
\end{picture}
\end{center}
}

Part (A) is commutative since $\rho_P^F$ is a right $D$-coaction, and 
Part (C) is commutative since the following diagram is so. 

{\setlength{\unitlength}{0.3mm}
\begin{center}
\begin{picture}(300, 130)(75,-55)
\put(50,50){$P$}
\put(150,50){$P\otimes P$}
\put(300,50){$P\otimes D\otimes P\otimes D$}
\put(90, 0){$P\otimes D$}
\put(200, 0){$P\otimes D$}
\put(34, -50){$P\otimes D$}
\put(325, -50){$D\otimes D$}
\put(85, 60){\small $\Delta_P\otimes \text{id}$}
\put(220, 60){\small $\rho_P^F\otimes\text{id}$}
\put(68, 25){\small $\cong$}
\put(175, 25){\small $\text{id}\otimes \varepsilon_P\otimes \text{id}$}
\put(150, -5){\small $\cong$}
\put(270, -5){\small $\rho_P^F\otimes \text{id}$}
\put(13, 0){\small $\rho_P^F\otimes \text{id}$}
\put(200, -43){\small $\cong$}
\put(347, 0){\small $\text{id}\otimes \varepsilon _P\otimes \text{id}$}
\put(65, 55){\vector(1,0){80}}
\put(192, 55){\vector(1,0){100}}
\put(62, 47.5){\vector(1,-1){35}}
\put(160, 47.5){\vector(-1,-1){35}}
\put(132, 5){\vector(1,0){60}}
\put(52.5, 45){\vector(0,-1){80}}
\put(342.5, 45){\vector(0,-1){80}}
\put(76, -45){\vector(1,0){245}}
\put(246, 2){\vector(2,-1){75}}
\end{picture}
\end{center}
}


\par 
The proof the commutativity of Part (B) is a little technical. 
For any $\boldsymbol{k}$-linear map $\gamma : P\longrightarrow \boldsymbol{k}$, the composition 
$$P\xrightarrow{\ \ \Delta_P\ \ }P\otimes P \xrightarrow{\ \gamma \otimes \text{id}_P\ } P$$
is a right $P$-comodule map, and hence it is also a right $C$-comodule map. 
Sending it by $F$ we have 
a right $D$-comodule map $(\gamma \otimes \text{id}_P)\circ \Delta_P: P\longrightarrow P$. 
Thus the following diagram commutes: 

{\setlength{\unitlength}{0.7cm}
\begin{center}
\begin{picture}(15,4.5)
\put(1,3.5){$P$}
\put(5,3.5){$P\otimes P$}
\put(11,3.5){$P$}
\put(1,0.5){$P$}
\put(4.2, 0.5){$P\otimes P\otimes D$}
\put(10.5, 0.5){$P\otimes D$}
\put(2.8, 3.8){\small $\Delta_P$}
\put(8.0, 3.8){\small $\gamma\otimes\text{id}$}
\put(1.8, 0.8){\small $\Delta_P\otimes \text{id}$}
\put(7.5, 0.8){\small $\gamma\otimes \text{id}\otimes \text{id}$}
\put(0.4, 2.2){\small $\rho_P^F$}
\put(11.4, 2.2){\small $\rho_P^F$}
\put(1.7, 3.6){\vector(1,0){3}}
\put(6.8, 3.6){\vector(1,0){4}}
\put(1.5, 0.6){\vector(1,0){2.4}}
\put(7,0.6){\vector(1,0){3.3}}
\put(1.2, 3.2){\vector(0,-1){2}}
\put(11.2, 3.2){\vector(0,-1){2}}
\end{picture}
\end{center}
}


Combining the above diagram with $\rho_P^F\circ (\gamma \otimes \text{id}_P)
= (\gamma \otimes \text{id}_{P\otimes D})\circ (\text{id}_P  \otimes \rho_P^F)$, 
we have 
$$
(\gamma \otimes \text{id}_{P\otimes D})\circ (\text{id}_P  \otimes \rho_P^F)\circ \Delta _P
=(\gamma \otimes \text{id}_{P\otimes D})\circ (\Delta _P\otimes \text{id}_P)\circ \rho_P^F.$$
Since this equation holds for all $\boldsymbol{k}$-linear maps $\gamma $, 
we see that $(\text{id}_P  \otimes \rho_P^F)\circ \Delta _P=(\Delta _P\otimes \text{id}_P)\circ \rho_P^F$. 
This implies the commutativity of Part (B). 

\par 
By the fundamental theorem for coalgebras, $C$ is a union of finite-dimensional subcoalgebras. 
Based on the fact, it can be shown that the coalgebra maps $\varphi _P : P \longrightarrow D$ for all finite-dimensional subcoalgebras $P$ induce a well-defined coalgebra map $\varphi :C \longrightarrow D$. 
In fact  it is easily verified that $(\varphi _Q)|_P=\varphi _P$ for two finite-dimensional subcoalgebras $P$ and $Q$ satisfying with $P\subset Q$. 
In this way, it is proved that there is a coalgebra map  $\varphi : C\longrightarrow D$ such that $\varphi |_P=\varphi _P$ for all finite-dimensional subcoalgebra $P$ of $C$. 
\par 
The coalgebra map $\varphi $ satisfies $F=\mathbb{M}^{\varphi }$. 
This fact can be derived as follows.  
\par \noindent 
$\bullet$ Let $(M, \rho_M)$ be a finite-dimensional right $C$-comodule. 
Then there is a finite-dimensional subcoalgebra $P$ of $C$ such that $\rho _M(M)\subset M\otimes P$. 
Since $(M\otimes P, \text{id}_M\otimes \rho_P)$ is a right $C$-comodule, we have a right $D$-comodule $(M\otimes P,\ (\text{id}_M\otimes \rho_P)^F)$. 
Since $M\otimes P$ is decomposed to a direct sum of finite copies of $P$ as a right $C$-comodule and 
$F$ is a $\boldsymbol{k}$-linear functor satisfying $U^C=U^D\circ F$, it follows that $(\text{id}_M\otimes \rho_P)^F=\text{id}_M\otimes \rho_P^F: M\otimes P \longrightarrow M\otimes P\otimes D$. 
\par 
Let $\rho_M^{\prime}: M\longrightarrow M\otimes P$ be the restriction of $\rho_M$. 
Then the map $\rho_M^{\prime}$ is a right $C$-comodule map from 
$(M, \rho_M)$ to $(M\otimes P, \text{id}\otimes \rho_P)$. 
Thus, 
$F(\rho_M^{\prime})=\rho_M^{\prime} : M\longrightarrow M\otimes P$ is a right $D$-comodule map from $(M, \rho_M^F)$ to $(M\otimes P, (\text{id}\otimes \rho_P)^F)=(M\otimes P, \text{id}_M\otimes \rho_P^F)$, and hence 
the following diagram commutes. 

{\setlength{\unitlength}{0.7cm}
\begin{center}
\begin{picture}(11,4.5)
\put(1,3.5){$M$}
\put(7,3.5){$M\otimes P$}
\put(1,0.5){$M\otimes D$}
\put(6.5, 0.5){$M\otimes P\otimes D$}
\put(4, 3.9){\small $\rho_M^{\prime}$}
\put(3.8, 0.9){\small $\rho_M^{\prime}\otimes \text{id}$}
\put(0.4, 2.2){\small $\rho_M^F$}
\put(8.1, 2.2){\small $\text{id}\otimes \rho_P^F$}
\put(1.7, 3.6){\vector(1,0){5}}
\put(2.9, 0.6){\vector(1,0){3.5}}
\put(1.2, 3.2){\vector(0,-1){2}}
\put(7.9, 3.2){\vector(0,-1){2}}
\end{picture}
\end{center}
}

\noindent 
It follows that we have the commutative diagram: 

{\setlength{\unitlength}{0.3mm}
\begin{center}
\begin{picture}(400,225)(50,-105)
\put(50, 100){$M$}
\put(350, 100){$M\otimes C$}
\put(165, 50){$M\otimes P$}
\put(150, 0){$M\otimes P\otimes D$}
\put(150, -50){$M\otimes \boldsymbol{k}\otimes D$}
\put(50, -100){$M\otimes D$}
\put(350, -100){$M\otimes D$}
\put(200, 110){\small $\rho_M$}
\put(120, 85){\small $\rho_M^{\prime}$}
\put(200, -93){\small $\text{id}$}
\put(290, 70){\small $\text{id}\otimes \iota_P$}
\put(290, -20){\small $\text{id}\otimes \varphi _P$}
\put(40, 10){\small $\rho_M^F$}
\put(375, 10){\small $\text{id}\otimes \varphi $}
\put(140, 30){\small $\text{id}\otimes \rho_P^F$}
\put(187, -22){\small $\text{id}\otimes \varepsilon_P\otimes \text{id}$}
\put(85, -35){\small $\rho_M^{\prime}\otimes \text{id}$}
\put(258, -65){\small $\cong $}
\put(66,104){\vector(1,0){280}}
\put(68,97){\vector(3,-1){95}}
\put(183, 42){\vector(0,-1){27}}
\put(183, -5){\vector(0,-1){32}}
\put(57.5, 94){\vector(0,-1){175}}
\put(73.5, -84){\vector(1,1){75}}
\put(96, -96){\vector(1,0){250}}
\put(215, -57){\vector(4,-1){128}}
\put(370, 93){\vector(0,-1){175}}
\put(213, 61){\vector(4,1){130}}
\put(213, 49){\vector(1,-1){135}}
\end{picture}
\end{center}
}

\noindent 
This implies that $\rho_M^F=(\text{id}\otimes \varphi )\circ \rho_M$, 
and hence we see that $F(M, \rho_M)=\mathbb{M}^{\varphi }(M, \rho_M)$. 

\par \noindent 
$\bullet$ Let $f : (M, \rho_M)\longrightarrow (N, \rho_N)$ be a right $C$-comodule map between finite-dimensional right $C$-comodules.  
Then 
\begin{align*}
F(f) =f & : (M, \rho_M^F)\longrightarrow (N, \rho_N^F),\\ 
\mathbb{M}^{\varphi }(f) =f& : (M, (\text{id}\otimes \varphi )\circ \rho_M)\longrightarrow (N, (\text{id}\otimes \varphi )\circ \rho_N). 
\end{align*}
As shown that $\rho_M^F=(\text{id}\otimes \varphi )\circ \rho_M$ and $\rho_N^F=(\text{id}\otimes \varphi )\circ \rho_N$, 
we see that 
$F(f) =\mathbb{M}^{\varphi }(f)$. Thus $F=\mathbb{M}^{\varphi }$ as $\boldsymbol{k}$-linear functors. 
\par 
Finally, we show the uniqueness of $\varphi $. 
Suppose that a coalgebra map $\psi : C\longrightarrow D$ satisfies $F=\mathbb{M}^{\psi}$, too. 
Let $P$ be a finite-dimensional subcoalgebra of $C$, and regard it as the right $C$-comodule $(P, \rho_P)$. 
Since $\mathbb{M}^{\varphi }(P, \rho_P)=\mathbb{M}^{\psi}(P, \rho_P)$, the following diagram commutes. 

{\setlength{\unitlength}{0.7cm}
\begin{center}
\begin{picture}(11,4.5)
\put(1,3.5){$P$}
\put(7,3.5){$P\otimes C$}
\put(1,0.5){$P\otimes C$}
\put(6.5, 0.5){$P\otimes D$}
\put(4.0, 3.8){\small $\rho_P$}
\put(4, 0.8){\small $\text{id}\otimes \varphi $}
\put(0.5, 2.2){\small $\rho_P$}
\put(7.9, 2.2){\small $\text{id}\otimes \psi$}
\put(1.7, 3.6){\vector(1,0){5}}
\put(2.8, 0.6){\vector(1,0){3.5}}
\put(1.2, 3.2){\vector(0,-1){2}}
\put(7.7, 3.2){\vector(0,-1){2}}
\end{picture}
\end{center}
}

\noindent 
Since the diagram 

{\setlength{\unitlength}{0.3mm}
\begin{center}
\begin{picture}(300, 130)(50,-75)
\put(50, 25){$P$}
\put(150, 25){$P\otimes C$}
\put(270, 25){$P\otimes D$}
\put(150, -25){$\boldsymbol{k}\otimes D$}
\put(270, -25){$\boldsymbol{k}\otimes D$}
\put(160, -40){\rotatebox[origin=c]{90}{$\cong$}}
\put(280, -40){\rotatebox[origin=c]{90}{$\cong$}}
\put(160, -56){$C$}
\put(280, -56){$D$}
\put(100, 35){\small $\rho_P$}
\put(210, 35){\small $\text{id}\otimes \varphi $}
\put(90, -20){\small $\iota_P$}
\put(125, 0){\small $\varepsilon_P\otimes \text{id}$}
\put(290, 0){\small $\varepsilon_P\otimes \text{id}$}
\put(210, -17){\small $\text{id}\otimes \varphi $}
\put(225, -47){\small $\varphi $}
\put(65, 30){\vector(1,0){80}}
\put(193, 30){\vector(1,0){70}}
\put(190, -22){\vector(1,0){72}}
\put(180, -52){\vector(1,0){92}}
\put(60, 15){\vector(3,-2){95}}
\put(165, 18){\vector(0,-1){30}}
\put(285, 18){\vector(0,-1){30}}
\end{picture}
\end{center}
}

\noindent 
commutes, and the same commutative diagram holds for $\psi$, we have 
$\varphi \circ \iota_P=\psi\circ \iota_P$. 
This implies that $\varphi =\psi$ since $C$ can be regarded as a union of finite-dimensional subcoalgebras. 
\end{proof} 

\par \medskip 
\begin{rem}\label{3-2}
The coalgebra map $\varphi _P: P\longrightarrow D$ in the above proof is a right $D$-comodule map from $(P, \rho_P^F)$ to $(D, \Delta _D)$. 
It follows from the following commutative diagram. 

{\setlength{\unitlength}{0.3mm}
\begin{center}
\begin{picture}(400, 175)(25,-85)
\put(55, 75){$P$}
\put(370, 75){$D$}
\put(40, -75){$P\otimes D$}
\put(350, -75){$D\otimes D$}
\put(140, 25){$P\otimes D$}
\put(275, 25){$\boldsymbol{k}\otimes D$}
\put(125, -25){$P\otimes D\otimes D$}
\put(260, -25){$\boldsymbol{k}\otimes D\otimes D$}
\put(210, 50){$\mathrm{(1)}$}
\put(210, -45){$\mathrm{(2)}$}
\put(85, 0){$\mathrm{(3)}$}
\put(210, 83){\small $\varphi _P$}
\put(210, 33){\small $\varepsilon_P\otimes \mathrm{id}$}
\put(205, -17){\small $\varepsilon_P\otimes \mathrm{id}$}
\put(200, -67){\small $\varphi _P\otimes \mathrm{id}$}
\put(43, 0){\small $\rho_P^F$}
\put(100, 58){\small $\rho_P^F$}
\put(330, 60){\small $\sim$}
\put(320, -40){\small $\sim$}
\put(105, -55){\small $\rho_P^F\otimes \mathrm{id}$}
\put(380, 0){\small $\Delta_D$}
\put(115, 5){\small $\mathrm{id}\otimes \Delta_D$}
\put(295, 5){\small $\mathrm{id}\otimes \Delta_D$}
\put(71, 78){\vector(1,0){290}}
\put(82, -72){\vector(1,0){260}}
\put(183, 28){\vector(1,0){85}}
\put(195, -22){\vector(1,0){58}}
\put(60, 65){\vector(0,-1){120}}
\put(375, 65){\vector(0,-1){120}}
\put(160, 20){\vector(0,-1){30}}
\put(290, 20){\vector(0,-1){30}}
\put(67, 70){\vector(2,-1){70}}
\put(75, -60){\vector(2,1){55}}
\put(300, 40){\line(2,1){65}}
\put(300, -32){\line(2,-1){58}}
\end{picture}
\end{center}
}

\noindent 
Here, commutativity of Parts $\mathrm{(1)}$ and $\mathrm{(2)}$ come from the definition of $\varphi _P$, and 
Part $\mathrm{(3)}$ comes from what $(P, \rho_P^F)$ is a right $D$-comodule. 
\end{rem}

\par \medskip 
As an application of Theorem~\ref{3-1} we have: 

\par \medskip 
\begin{cor}\label{3-3}
Let $C, D$ be two coalgebras over $\boldsymbol{k}$, and $F: \mathbb{M}^C \longrightarrow \mathbb{M}^D$ be a $\boldsymbol{k}$-linear functor. 
If $F$ is an equivalence of $\boldsymbol{k}$-linear categories and $U^D\circ F=U^C$ is satisfied, then 
the coalgebra map $\varphi : C\longrightarrow D$ determined by $F=\mathbb{M}^{\varphi }: \mathbb{M}^C \longrightarrow \mathbb{M}^D$ in Theorem~\ref{3-1} is an isomorphism. 
\end{cor} 

\par \medskip 
The above corollary can be easily proved by taking a quasi-inverse $G$ of $F$ such as $U^C\circ G=U^D$, $G\circ F=1_{\mathbb{M}^C}$ and $F\circ G=1_{\mathbb{M}^D}$. 

\par 
Dualizing of Theorem~\ref{3-1} we have also the following corollary. 

\par \medskip 
\begin{cor}
Let $A, B$ be two finite-dimensional algebras over $\boldsymbol{k}$, and $F: {}_B\mathbb{M} \longrightarrow {}_A\mathbb{M}$ be a  $\boldsymbol{k}$-linear functor. If ${}_AU\circ F={}_BU$ is satisfied for the forgetful functors ${}_AU, {}_BU$ to $\text{Vect}_{\boldsymbol{k}}^{\mathrm{f.d.}}$, then 
\begin{enumerate}
\item[$(1)$] there is a unique algebra map $\varphi : A\longrightarrow B$ such that 
$F={}_{\varphi }\mathbb{M}: {}_B\mathbb{M} \longrightarrow {}_A\mathbb{M}$, where ${}_{\varphi }\mathbb{M}$ stands for the $\boldsymbol{k}$-linear functor induced from $\varphi $, 
\item[$(2)$] if $F$ is an equivalence of $\boldsymbol{k}$-linear categories, then the algebra map $\varphi : A\longrightarrow B$ given by 
$(1)$ is an isomorphism. 
\end{enumerate}
\end{cor} 

\par \medskip 
Now, we show the main theorem in the present paper: 

\par \medskip 
\begin{thm}\label{3-5}
Let $A, B$ be weak bialgebras over $\boldsymbol{k}$, and 
$F=(F, \bar{\phi}^F, \bar{\omega}^F) : \mathbb{M}^A \longrightarrow \mathbb{M}^B$ be a strong $\boldsymbol{k}$-linear comonoidal functor. 
Suppose that $U^B\circ F=U^A$ as $\boldsymbol{k}$-linear monoidal categories. 
Then there is a unique weak bialgebra map $\varphi : A\longrightarrow B$ such that $F=\mathbb{M}^{\varphi}$ as $\boldsymbol{k}$-linear monoidal categories and $\bar{\omega}^F=\varphi |_{A_s}: A_s \longrightarrow B_s$ is an algebra isomorphism. 
Furthermore, $\hat{U}^B\circ F={}_{\varphi_s^{-1}}\Mod _{\varphi_s^{-1}}\circ \hat{U}^A$ is satisfied, where 
$\hat{U}^A: \mathbb{M}^A \longrightarrow {}_{A_s}\Mod_{A_s},\ \hat{U}^B: \mathbb{M}^B \longrightarrow {}_{B_s}\Mod_{B_s}$ are forgetful functors. 
\end{thm} 
\begin{proof} 
By Theorem~\ref{3-1} there is a unique coalgebra map  $\varphi : A\longrightarrow B$ such that  $F=\mathbb{M}^{\varphi}$ as $\boldsymbol{k}$-linear functors. 
Since $U^{B}\circ F=U^{A}$ as $\boldsymbol{k}$-linear monoidal functors, for all $M, N\in \mathbb{M}^{A}$ the composition 
\begin{align*}
U^B\bigl( F(M)\bigr) \otimes U^B\bigl( F(N)\bigr) &\xrightarrow{\ \phi^{U^B}_{F(M), F(N)}\ } 
U^B\bigl( F(M)\otimes _{B_s}F(N)\bigr)  \\ 
&\qquad\qquad \xrightarrow{\ U^B\bigl( (\bar{\phi}^F_{M,N})^{-1}\bigr) \ } 
U^B\bigl( F(M\otimes _{A_s}N)\bigr) 
\end{align*}
coincides with the natural projection $\phi ^{U^A}_{M,N}: U^A(M)\otimes U^A(N) \longrightarrow U^A(M\otimes _{A_s}N)$. 
This implies that the diagram 

{\setlength{\unitlength}{0.7cm}
\begin{center}
\begin{picture}(11,4.5)
\put(0.2, 3.5){$M\otimes N$}
\put(7,3.5){$M\otimes N$}
\put(-0.3,0.5){$M\otimes _{A_s}A$}
\put(6.8, 0.5){$M\otimes _{B_s}N$}
\put(3.8, 3.8){\small $\text{id}_{M\otimes N}$}
\put(3.8, 0.9){\small $\bar{\phi}^F_{M,N}$}
\put(-1.8, 2.2){\small natural proj.}
\put(8.2, 2.2){\small natural proj.}
\put(2.2, 3.6){\vector(1,0){4.5}}
\put(2.1, 0.6){\vector(1,0){4.5}}
\put(1.2, 3.2){\vector(0,-1){2}}
\put(7.9, 3.2){\vector(0,-1){2}}
\end{picture}
\end{center}
}

\noindent 
commutes. 
This means that the map $\bar{\phi}^F_{M,N} : F(M\otimes _{A_s}N)\longrightarrow F(M)\otimes _{B_s}F(N)$ 
is induced from the identity map $\text{id}_{M\otimes N}$. 
\par 
Let us show that $\bar{\omega}^F=\varphi |_{A_s}: A_s \longrightarrow B_s$ is an algebra isomorphism. 
Since $U^{B}\circ F=U^{A}$ as $\boldsymbol{k}$-linear monoidal functors, the composition 
$$U^{B}\bigl( (\bar{\omega}^F)^{-1}\bigr) \circ \omega ^{U^{B}} : \boldsymbol{k}\longrightarrow (U^{B}\circ F)(A_s)=A_s$$
coincides with $\omega ^{U^{A}}: \boldsymbol{k}\longrightarrow A_s$. 
Thus $(\bar{\omega}^F)^{-1}\circ \omega ^{U^{B}}= \omega ^{U^{A}}$ as maps, and hence $\bar{\omega }^F(1)=1$. 
\par 
Since $\bar{\omega}^F: F(A_s)\longrightarrow B_s$ is a right $B$-comodule map and $F(A_s)=\mathbb{M}^{\varphi}(A_s)$, 
the following diagram commutes. 

{\setlength{\unitlength}{0.7cm}
\begin{center}
\begin{picture}(8.5,7)
\put(0.2, 6){$M\otimes N$}
\put(7, 6){$M\otimes N$}
\put(0, 3){$M\otimes _{A_s}A$}
\put(0, 0.5){$M\otimes _{A_s}A$}
\put(6.8, 0.5){$M\otimes _{B_s}N$}
\put(4, 6.3){\small $\bar{\omega}^F$}
\put(3.8, 0.8){\small $\bar{\omega}^F\otimes \text{id}$}
\put(-0.5, 4.7){\small $\Delta_A|_{A_s}$}
\put(-0.5, 1.9){\small $\text{id}\otimes \varphi$}
\put(8.1, 3.2){\small $\Delta_B|_{B_s}$}
\put(2.2, 6.1){\vector(1,0){4.5}}
\put(2.3, 0.6){\vector(1,0){4.3}}
\put(1, 5.7){\vector(0,-1){2}}
\put(1, 2.5){\vector(0,-1){1.5}}
\put(7.9, 5.7){\vector(0,-1){4.7}}
\end{picture}
\end{center}
}

\noindent 
Thus $\Delta _B\bigl( \bar{\omega }^F(y)\bigr) =\bar{\omega}^F(y_{(1)})\otimes \varphi (y_{(2)})$ for all $y\in A_s$. 
Applying $\varepsilon_B\otimes \text{id}$ to both sides, we have 
\begin{align*}
\bar{\omega }^F(y)
&=\varepsilon_B\bigl( \bar{\omega}^F(y_{(1)})\bigr) \varphi (y_{(2)}) \\ 
&=\varepsilon_A(y_{(1)}) \varphi (y_{(2)})\\ 
&=\varphi (y). 
\end{align*}
This implies that $\bar{\omega }^F=\varphi |_{A_s} : A_s\longrightarrow B_s$, and hence $\varphi (1)=\bar{\omega }^F(1)=1$. 
\par 
Next we show that $\bar{\omega }^F$ preserves products. 
Since $F=(F, \bar{\phi}^F, \bar{\omega}^F)$ is a comonoidal functor, the following diagram commutes for all $M\in \mathbb{M}^{A}$: 

{\setlength{\unitlength}{0.7cm}
\begin{center}
\begin{picture}(12,4.5)
\put(0.2,3.5){$F(A_s\otimes _{A_s}M)$}
\put(9,3.5){$F(M)$}
\put(-0.2, 0.5){$F(A_s)\otimes_{B_s} F(M)$}
\put(8.5, 0.5){$B_s\otimes_{B_s} F(M)$}
\put(5.5, 3.9){\small $F(l_M^{A})$}
\put(5.1, 0.8){\small $\bar{\omega} ^F\otimes_{B_s} \text{id}$}
\put(0.3, 2){\small $\bar{\phi}^F_{A_s, M}$}
\put(9.9, 2){\small $l^{B}_M$}
\put(3.7, 3.6){\vector(1,0){5}}
\put(4.2, 0.6){\vector(1,0){4}}
\put(1.8, 3.1){\vector(0,-1){2}}
\put(9.7, 1.1){\vector(0,1){2}}
\end{picture}
\end{center}
}

\noindent 
In particular, the following diagram commutes: 

{\setlength{\unitlength}{0.7cm}
\begin{center}
\begin{picture}(11,7)
\put(-0.7, 6){$F(A_s\otimes _{A_s}A_s)$}
\put(8, 6){$F(A_s)\otimes _{B_s}F(A_s)$}
\put(0.2, 3){$F(A_s)$}
\put(9, 3){$B_s\otimes _{B_s}F(A_s)$}
\put(0.7, 0.5){$B_s$}
\put(8.8, 0.5){$B_s\otimes _{B_s}B_s$}
\put(4.8, 6.4){\small $\bar{\phi }_{A_s, A_s}^F$}
\put(4.8, 3.6){\small $l_{F(A_s)}^{B}$}
\put(4.9, 0.9){\small $l_{B_s}^{B}$}
\put(-0.5, 4.7){\small $F(l_{A_s}^{A})$}
\put(10.2, 4.7){\small $\bar{\omega}^F\otimes _{B_s}\text{id}$}
\put(0.2, 1.9){\small $\bar{\omega }^F$}
\put(10.1, 1.9){\small $\text{id}\otimes _{B_s}\bar{\omega}^F$}
\put(2.8, 6.1){\vector(1,0){5}}
\put(8.7, 3.2){\vector(-1,0){6.7}}
\put(8.5, 0.6){\vector(-1,0){7}}
\put(1, 5.7){\vector(0,-1){2}}
\put(1, 2.7){\vector(0,-1){1.5}}
\put(9.9, 5.5){\vector(0,-1){1.7}}
\put(9.9, 2.7){\vector(0,-1){1.7}}
\end{picture}
\end{center}
}

\noindent 
It follows that 
$$\bar{\omega }^F(y_1y_2)=\bar{\omega }^F(y_1)\bar{\omega }^F(y_2)$$
for all $y_1, y_2\in A_s$ 
since $\bar{\phi }_{A_s, A_s}^F$ is induced from the identity map $\text{id}_{A_s\otimes A_s}$. 
Therefore, it is shown that $\bar{\omega }^F=\varphi|_{A_s}$ is an algebra isomorphism. 
\par 
Next let us show that $\varphi$ is a weak bialgebra map. 
For this it is enough to show that $\varphi$ preserves products. 
For two subspaces $P, P^{\prime}$ of $A$, let $PP^{\prime}$ denote the subspace of $A$ spanned by the set $\{ \ pp^{\prime}\ | \ p\in P,\ p^{\prime}\in P^{\prime}\ \}$. 
Let $\mu _{P,P^{\prime}}: P\otimes P^{\prime} \longrightarrow PP^{\prime}$ be the restriction of the product $\mu_A$ of $A$. 
Then $\mu _{P,P^{\prime}}$ induces a $\boldsymbol{k}$-linear map $\bar{\mu}_{P, P^{\prime}}: P\otimes _{A_s}P^{\prime} \longrightarrow PP^{\prime}$ since the equation $(p\cdot y)p^{\prime}=p(y\cdot p^{\prime})$ holds for $p\in P,\ y\in A_s,\ p^{\prime}\in P^{\prime}$. 
\par 
Let $a, a^{\prime}\in A$, and let $P, P^{\prime}$ be finite-dimensional subcoalgebras of $A$ such that $a\in P$,\ $a^{\prime}\in P^{\prime}$. 
Then $PP^{\prime}$ is also a finite-dimensional subcoalgebra of $A$ containing $aa^{\prime}$, and the map 
$\bar{\mu}_{P, P^{\prime}}$ is a right $A$-comodule map. 
Let $\varphi_P, \varphi_{P^{\prime}}$ be the coalgebra maps defined in the proof of Theorem~\ref{3-1}. 
We will show that the following diagram commutes. 

\begin{equation}\label{eq3-1}
{\setlength{\unitlength}{0.3mm}
\begin{picture}(300, 40)(10, 0)
\put(-10, 25){$F(P\otimes_{A_s} P^{\prime})$}
\put(125, 25){$F(P)\otimes_{B_s} F(P^{\prime})$}
\put(285, 25){$B\otimes_{B_s} B$}
\put(0, -25){$F(PP^{\prime})$}
\put(300, -25){$B$}
\put(80, 35){\small $\bar{\phi}_{P,P^{\prime}}^F$}
\put(222, 35){\small $\varphi_P\otimes_{B_s} \varphi_{P^{\prime}}$}
\put(-22, 5){\small $F(\bar{\mu}_{P, P^{\prime}})$}
\put(310, 4){\small $\bar{\mu}_B$}
\put(155, -17){\small $\varphi_{PP^{\prime}}$}
\put(67, 28){\vector(1,0){50}}
\put(222, 28){\vector(1,0){55}}
\put(50, -23){\vector(1,0){240}}
\put(25, 20){\vector(0,-1){30}}
\put(305, 20){\vector(0,-1){30}}
\end{picture}
}
\end{equation}


\vspace{1cm}\noindent 
Here, $\bar{\mu}_B$ is the induced map from the product $\mu_B$ of $B$. 
Since $\varphi_P, \varphi_{P^{\prime}}$ are right $B$-comodule maps by Remark~\ref{3-2}, the map $\varphi_P\otimes_{B_s} \varphi_{P^{\prime}}$ is well-defined. 
Since $F=\Mod^{\varphi }$ and $\hat{U}^B\circ \Mod^{\varphi }={}_{\varphi_s^{-1}}\Mod_{\varphi_s^{-1}}\circ\hat{U}^A$ as $\boldsymbol{k}$-linear monoidal functors, 
one can verify that $\bar{\phi}_{P,P^{\prime}}^F: F(P\otimes_{A_s} P^{\prime})\longrightarrow F(P)\otimes_{B_s} F(P^{\prime})$ is a $\boldsymbol{k}$-linear map given by 
\begin{equation}\label{eq3-2}
\bar{\phi}_{P,P^{\prime}}^F(p\otimes _{A_s}p^{\prime})=p\otimes _{B_s}p^{\prime}\qquad (p\in F(P)=P,\ p^{\prime}\in F(P^{\prime})=P^{\prime}). 
\end{equation}

Let $(F(P)\otimes _{B_s}F(P^{\prime}),\ \rho _{P,P^{\prime}}^F)$ be the tensor product in $\Mod^B$ of right comodules $(F(P), \rho_P^F)$ and $(F(P^{\prime}), \rho_{P^{\prime}}^F)$. 
Here, $\rho _{P,P^{\prime}}^F$ is the right coaction given by 
$$\rho _{P,P^{\prime}}^F(m\otimes _{B_s}n)=(m_{(0)}\otimes _{B_s}n_{(0)})\otimes m_{(1)}n_{(1)}\qquad (m\in F(P),\ n\in F(P^{\prime})), $$
and $\rho_P^F(m)=m_{(0)}\otimes m_{(1)},\ \rho_{P^{\prime}}^F(n)=n_{(0)}\otimes n_{(1)}$. 
Since $\bar{\phi}_{P,P^{\prime}}^F$ is a right $B$-comodule map from $(F(P\otimes_{A_s}P^{\prime}), \ \rho_{P\otimes_{A_s}P^{\prime}}^F)$ to $(F(P)\otimes _{B_s}F(P^{\prime}),\ \rho _{P,P^{\prime}}^F)$, the following diagram commutes. 

{\setlength{\unitlength}{0.7cm}
\begin{center}
\begin{picture}(13,4.5)
\put(0.2,3.5){$F(P\otimes _{A_s}P^{\prime})$}
\put(9,3.5){$F(P\otimes _{A_s}P^{\prime})\otimes B$}
\put(-0.2, 0.5){$F(P)\otimes _{B_s}F(P^{\prime})$}
\put(8.5, 0.5){$\bigl( F(P)\otimes _{B_s}F(P^{\prime})\bigr) \otimes B$}
\put(5.5, 3.9){\small $\rho_{P\otimes _{A_s}P^{\prime}}^F$}
\put(5.6, 0.9){\small $\rho _{P,P^{\prime}}^F$}
\put(0.5, 2.2){\small $\bar{\phi }_{P,P^{\prime}}^F$}
\put(11, 2.2){\small $\bar{\phi }_{P,P^{\prime}}^F\otimes \text{id}_B$}
\put(3.7, 3.6){\vector(1,0){5}}
\put(4.2, 0.6){\vector(1,0){4}}
\put(1.8, 3.1){\vector(0,-1){2}}
\put(10.7, 3.1){\vector(0,-1){2}}
\end{picture}
\end{center}
}


On the other hand, there is a $(B_s, B_s)$-bimodule map $\chi _{P,P^{\prime}} : (P\otimes B)\otimes _{B_s}(P^{\prime}\otimes B) \longrightarrow (P\otimes _{B_s}P^{\prime})\otimes B$, which is defined by 
$$\chi _{P, P^{\prime}}\bigl( (p\otimes b)\otimes _{B_s}(p^{\prime}\otimes b^{\prime})\bigr) 
=(p\otimes _{B_s}p^{\prime})\otimes bb^{\prime}\qquad (p\in P,\ p^{\prime}\in P^{\prime},\ b, b^{\prime}\in B).$$
We also define a $\boldsymbol{k}$-linear map $\varepsilon_{P, P^{\prime}}: P\otimes _{A_s}P^{\prime}\longrightarrow \boldsymbol{k}$ by 
$$\varepsilon_{P, P^{\prime}}(p\otimes _{A_s}p^{\prime})=\varepsilon (p1_{(1)})\varepsilon (1_{(2)}p^{\prime})\qquad (p\in P,\ p^{\prime}\in P^{\prime}).$$
Then it can be shown that the following diagram commutes. 

{\setlength{\unitlength}{0.3mm}
\begin{center}
\begin{picture}(500, 245)(5,-115)
\put(0, 100){$F(P\otimes _{A_s}P^{\prime})$}
\put(140, 100){$F(P)\otimes _{B_s}F(P^{\prime})$}
\put(450, 100){$B\otimes _{B_s}B$}
\put(125, 50){$F(P)\otimes _{B_s}F(P^{\prime})\otimes B$}
\put(285, 50){$(F(P)\otimes B)\otimes _{B_s}(F(P^{\prime})\otimes B)$}
\put(135, 0){$F(P\otimes _{A_s}P^{\prime})\otimes B$}
\put(340, 0){$\boldsymbol{k}\otimes B$}
\put(150, -50){$F(PP^{\prime})\otimes B$}
\put(13, -100){$F(PP^{\prime})$}
\put(475, -100){$B$}
\put(420, 0){($\ast $)}
\put(100, 110){\small $\bar{\phi}_{P,P^{\prime}}^F$}
\put(330, 110){\small $\varphi_P\otimes _{B_s}\varphi_{P^{\prime}}$}
\put(65, 42){\small $\rho_{P\otimes _{A_s}P^{\prime}}^F$}
\put(185, 80){\small $\rho_{P,P^{\prime}}^F$}
\put(186, 25){\small $\bar{\phi}_{P,P^{\prime}}^F\otimes \text{id}$}
\put(255, 59){\small $\chi_{P,P^{\prime}}$}
\put(230, 83){\small $\rho_P^F\otimes _{B_s}\rho_{P^{\prime}}^F$}
\put(320, 80){\small $(F(\varepsilon_P)\otimes \text{id})_{B_s}(F(\varepsilon_{P^{\prime}})\otimes \text{id})$}
\put(250, 8){\small $F(\varepsilon_{P,P^{\prime}})\otimes \text{id}$}
\put(400, -42){\small $\cong$}
\put(185, -22){\small $F(\bar{\mu}_{P,P^{\prime}})\otimes \text{id}$}
\put(280, -38){\small $F(\varepsilon_{PP^{\prime}})\otimes \text{id}$}
\put(90, -60){\small $\rho _{PP^{\prime}}^F$}
\put(-7, 0){\small $F(\bar{\mu} _{P,P^{\prime}})$}
\put(250, -90){\small $\varphi_{PP^{\prime}}$}
\put(484, 0){\small $\bar{\mu }_{B}$}
\put(77, 103){\vector(1,0){55}}
\put(237, 103){\vector(1,0){200}}
\put(62, -97){\vector(1,0){400}}
\put(41, 88){\vector(0,-1){168}}
\put(480, 88){\vector(0,-1){168}}
\put(180, 90){\vector(0,-1){28}}
\put(180, 12){\vector(0,1){28}}
\put(180, -9){\vector(0,-1){26}}
\put(65, 88){\vector(1,-1){70}}
\put(67, -88){\vector(2,1){75}}
\put(225, -45){\vector(3,1){110}}
\put(240, 3){\vector(1,0){90}}
\put(281, 53){\vector(-1,0){33}}
\put(380, -5){\vector(1,-1){82}}
\put(240, 98){\vector(2,-1){62}}
\put(370, 63){\vector(2,1){72}}
\end{picture}
\end{center}
}

\noindent 
In fact, the commutativity of ($\ast $) in the above diagram comes from the following three facts: 
\begin{enumerate}
\item[(i)] $\bar{\phi }^F_{M,N}: F(M\otimes _{A_s}N)\longrightarrow F(M)\otimes _{B_s}F(N)$ is an isomorphism, 
\item[(ii)] the equation \eqref{eq3-2} holds, 
\item[(iii)] for $p\in P,\ p^{\prime}\in P^{\prime}$, 
\begin{enumerate}
\item[$\bullet$] 
$p\otimes _{A_s}p^{\prime}=(p\cdot 1)\otimes _{A_s}p^{\prime}
=\bigl( p\cdot 1_{(1)}\varepsilon _s(1_{(2)})\bigr) \otimes _{A_s}p^{\prime}
=p1_{(1)}\otimes _{A_s}\varepsilon_s(1_{(2)})p^{\prime}$,  
\item[$\bullet$] $\varepsilon \bigl(\varepsilon_s(1_{(2)})p^{\prime}\bigr) =\varepsilon (1_{(2)}p^{\prime})$ by Lemma~\ref{2-2}(2).
\end{enumerate}
\end{enumerate}

\noindent 
Thus, the diagram \eqref{eq3-1} commutes, and therefore $\varphi$ preserves products. 
\par 
To complete the proof we need to verify that $F=\mathbb{M}^{\varphi}$ as $\boldsymbol{k}$-linear comonoidal functors. 
This is an easy consequence of the proof of Lemma~\ref{2-7}(2) since $\bar{\phi}^F_{M,N}$ is induced from the identity map $\text{id}_{M\otimes N}$ for all $M, N\in \mathbb{M}^{A}$. 
\end{proof} 

\par \medskip 
By Corollary~\ref{3-3} and Theorem~\ref{3-5} we have: 

\par \medskip 
\begin{cor}\label{3-6}
Let $A, B$ be weak bialgebras over $\boldsymbol{k}$, and 
$F=(F, \bar{\phi}^F, \bar{\omega}^F) : \mathbb{M}^A \longrightarrow \mathbb{M}^B$ be a strong $\boldsymbol{k}$-linear comonoidal functor satisfying $U^B\circ F=U^A$ as $\boldsymbol{k}$-linear monoidal categories. 
If $F$ is a $\boldsymbol{k}$-linear monoidal equivalence, then 
the weak bialgebra map $\varphi : A\longrightarrow B$ determined in Theorem~\ref{3-5} is an isomorphism of weak bialgebras. 
\end{cor} 

\par \bigskip 
\section{Categorical aspects of indecomposable weak bialgebras}
\par 
Let $A=(A, \Delta _A, \varepsilon _A)$ and $B=(B, \Delta _B, \varepsilon _B)$ be two weak bialgebras over $\boldsymbol{k}$, and set $H=A\oplus B$. 
Then, $H$ has a weak bialgebra structure such that 
the algebra structure is the product and the coalgebra structure is the direct sum of $A$ and $B$. 
The target and source counital maps $\varepsilon_t$ and $\varepsilon_s$ are given by 
\begin{align*}
\varepsilon_t(x)&=(\varepsilon _A)_t(a)+(\varepsilon _B)_t(b),\\ 
\varepsilon_s(x)&=(\varepsilon _A)_s(a)+(\varepsilon _B)_s(b)
\end{align*}
for all $x=a+b\in H,\ a\in A,\ b\in B$, 
where $(\varepsilon _A)_t, (\varepsilon _A)_s$ are the target and source counital maps of $A$, and $(\varepsilon _B)_t, (\varepsilon _B)_s$ are that of $B$, respectively. 
Moreover, the target and source subalgebras are given as follows: 
\begin{align*}
H_t & =\varepsilon _t(H)=(\varepsilon _A)_t(A)+(\varepsilon _B)_t(B)=A_t+B_t,\\ 
H_s & =\varepsilon _s(H)=(\varepsilon _A)_s(A)+(\varepsilon _B)_s(B)=A_s+B_s. 
\end{align*}

\noindent 
We call the above weak bialgebra $H$ the direct sum of $A$ and $B$. 
\par 
A weak bialgebra $H$ is called \textit{indecomposable} if there are no weak bialgebras $A$ and $B$ such that $H\cong A\oplus B$. 
Any finite-dimensional weak bialgebra can be decomposed into the finitely many direct sum of indecomposable ones. More precisely, we have: 

\par \medskip 
\begin{thm}
Let $H$ be a finite-dimensional weak bialgebra over $\boldsymbol{k}$. Then
\begin{enumerate}
\item[$(1)$] there are finitely many indecomposable weak bialgebras $H_i \ (i=1,\ldots , k)$ such that 
$H=H_1\oplus \cdots \oplus H_k$. 
\item[$(2)$] If  indecomposable weak bialgebras $H_1, \ldots , H_k$ and $H_1^{\prime}, \ldots , H_l^{\prime}$ satisfy 
$$H_1\oplus \cdots \oplus H_k=H=H_1^{\prime}\oplus \cdots \oplus H_l^{\prime},$$
then $k=l$, and there is a permutation $\sigma \in \mathfrak{S}_k$ such that $H_i^{\prime}=H_{\sigma (i)}$ for all $i=1, \ldots , k$. 
\end{enumerate}
\end{thm}

\par \medskip 
The above theorem can be proved by decomposability and uniqueness of finite-dimensional algebras into indecomposable ones (see \cite{Wakui_ContempMath} for detail). 

\par 
A $\boldsymbol{k}$-linear monoidal category $\mathcal{C}$ is called \textit{indecomposable} if there are no 
$\boldsymbol{k}$-linear monoidal categories $\mathcal{C}_1, \mathcal{C}_2$ such that $\mathcal{C}\simeq \mathcal{C}_1\times \mathcal{C}_2$ as $\boldsymbol{k}$-linear monoidal categories. 

\par \medskip 
\begin{prop}\label{4-2}
Let $H$ be a weak bialgebra over $\boldsymbol{k}$. 
If $H$ is decomposable, then 
\begin{enumerate}
\item[$(1)$] the left $H$-module category ${}_H\Mod$ and 
 the finite-dimensional left $H$-module category ${}_H\mathbb{M}$ are decomposable. 
\item[$(2)$] the right $H$-comodule category $\Mod^H$ and  
 the finite-dimensional right $H$-comodule category $\mathbb{M}^H$ are decomposable. 
\end{enumerate}
\end{prop} 
\begin{proof}
Suppose that $H=A\oplus B$ for some weak bialgebras $A, B$ over $\boldsymbol{k}$. 
\par 
(1) The left $H$-module category ${}_H\Mod$ is equivalent to the Cartesian product ${}_A\Mod\times {}_B\Mod$ as $\boldsymbol{k}$-linear monoidal categories. 
An equivalence is given by the following monoidal functors, which are quasi-inverse each other:
\begin{align*}
(F, \phi ^F, \omega ^F) : & {}_H\Mod \longrightarrow {}_A\Mod\times {}_B\Mod,\\ 
F(X)&=(1_A\cdot X,\ 1_B\cdot X), \\ 
\phi^F_{X,Y}&=\mathrm{id}_{F(X)\circledast F(Y)}: F(X)\circledast F(Y) \longrightarrow F(X\circledast Y), \\ 
\omega ^F &=\mathrm{id}_{(A_t, B_t)} :  (A_t, B_t) \longrightarrow F(H_t),\\ 
(G, \phi ^G, \omega ^G) : & {}_A\Mod\times {}_B\Mod \longrightarrow {}_H\Mod,\\ 
G(U, V)&=U\times V,\\ 
\phi^G_{(U_1, V_1), (U_2, V_2)}&=\mathrm{id}_{G(U_1, V_1)\circledast G(U_2, V_2)}: G(U_1, V_1)\circledast G(U_2, V_2) \longrightarrow G\bigl( (U_1, V_1)\circledast (U_2, V_2)\bigr) ,\\ 
\omega ^G &: H_t \longrightarrow G(A_t, B_t)=A_t\times B_t,\ \ 
\omega ^G(z)=(1_A\cdot z,\ 1_B\cdot z)\qquad (z\in H_t). 
\end{align*}

By restricting the above equivalence ${}_H\Mod\simeq {}_A\Mod\times {}_B\Mod$ to finite dimension a $\boldsymbol{k}$-linear monoidal equivalence ${}_H\mathbb{M}\simeq {}_A\mathbb{M}\times {}_B\mathbb{M}$ is obtained. 
\par 
(2) The right $H$-comodule category $\Mod^H$ is equivalent to the Cartesian product $\Mod^A\times \Mod^B$ as $\boldsymbol{k}$-linear monoidal categories. 
An equivalence is given by the following monoidal functors, which are quasi-inverse each other: 
\begin{align*}
(F, \phi ^F, \omega ^F) : & \Mod^H \longrightarrow \Mod^A \times \Mod^B,\\ 
F\bigl( (X, \rho )\bigr) &=\Bigl( \bigl( (\varepsilon_A\circ \pi_A)\cdot X,\ (\mathrm{id}\otimes \pi _A)\circ \rho |_{(\varepsilon_A\circ \pi_A)\cdot X}\bigr) , \bigl( (\varepsilon_B\circ \pi_B)\cdot X,\ (\mathrm{id}\otimes \pi _B)\circ \rho |_{(\varepsilon_B\circ \pi_B)\cdot X}\bigr) \Bigr)\\ 
\phi^F_{X,Y}&=\mathrm{id}_{F(X)\circledast F(Y)}: F(X)\circledast F(Y) \longrightarrow F(X\circledast Y), \\ 
\omega ^F &=\mathrm{id}_{(A_s, B_s)} :  (A_s, B_s) \longrightarrow F(H_s),\\ 
(G, \phi ^G, \omega ^G) : & \Mod^A\times\Mod^B \longrightarrow \Mod^H,\\ 
G(U, V)&=U\times V,\\ 
\phi^G_{(U_1, V_1), (U_2, V_2)}&=\mathrm{id}_{G(U_1, V_1)\circledast G(U_2, V_2)}: G(U_1, V_1)\circledast G(U_2, V_2) \longrightarrow G\bigl( (U_1, V_1)\circledast (U_2, V_2)\bigr) ,\\ 
\omega ^G &: H_s \longrightarrow G(A_s, B_s)=A_s\times B_s,\ \ 
\omega ^G(y)=\bigl(\pi_A(y),\ \pi_B(y)\bigr) \qquad (y\in H_s), 
\end{align*}
where $\pi_A$ and $\pi_B$ are natural surjections associated with the direct sum decomposition $H=A\oplus B$. 
\par 
By restricting the above equivalence $\Mod^H\simeq \Mod^A\times \Mod^B$ to finite dimension a $\boldsymbol{k}$-linear monoidal equivalence $\mathbb{M}^H\simeq \mathbb{M}^A\times \mathbb{M}^B$ is obtained. 
\end{proof}

\par \medskip 
The converse of the above proposition is true. 
To prove it we need the following reconstruction theorem of bialgebras. 

\par \medskip 
\begin{thm}[{\bf Ulbrich~\cite{Ulbrich}, Schauenburg~\cite[Theorem 5.4]{Schauen2}}]\label{4-3}
Let $\mathcal{C}$ be a $\boldsymbol{k}$-linear monoidal category, and 
$\omega : \mathcal{C}\longrightarrow \text{Vect}_{\boldsymbol{k}}^{\text{f.d.}}$ be a faithful and exact $\boldsymbol{k}$-linear monoidal functor. 
Then, there are a bialgebra $B$ and a monoidal equivalence $F: \mathcal{C}\longrightarrow \mathbb{M}^B$ such that 
$U^B\circ F=\omega $, where $U^B: \mathbb{M}^B\longrightarrow \text{Vect}_{\boldsymbol{k}}^{\text{f.d.}}$ is the forgetful functor. 
\qed 
\end{thm} 

\par \medskip 
Combining Theorems~\ref{3-5} and \ref{4-3} one can show the following.

\par \medskip 
\begin{thm}\label{4-4}
Let $H$ be a finite-dimensional weak bialgebra over $\boldsymbol{k}$. Then 
$H$ is indecomposable as a weak bialgebra if and only if the finite-dimensional left $H$-module category ${}_H\mathbb{M}$ is indecomposable as a $\boldsymbol{k}$-linear monoidal category. 
\end{thm} 
\begin{proof}
By the contraposition of Proposition~\ref{4-2} ``if part" holds. 
We will show that ``only if part". 
Suppose that $H$ is indecomposable whereas ${}_H\mathbb{M}$ is not. 
Then, there are $\boldsymbol{k}$-linear monoidal categories $\mathcal{C}_1, \mathcal{C}_2$ such that ${}_H\mathbb{M}\simeq \mathcal{C}_1\times \mathcal{C}_2$ as $\boldsymbol{k}$-linear monoidal categories. 
Let ${}_HU: {}_H\mathbb{M}\longrightarrow \text{Vect}_{\boldsymbol{k}}^{\text{f.d.}}$ be the forgetful functor, and 
$F: \mathcal{C}_1\times \mathcal{C}_2\longrightarrow {}_H\mathbb{M}$ be a $\boldsymbol{k}$-linear monoidal category equivalence. 
Since the two $\boldsymbol{k}$-linear monoidal functors 
\begin{align*}
\omega _1 :  & \ \mathcal{C}_1 \cong \mathcal{C}_1\times 0 \xrightarrow{\ \ F\ \ }  {}_H\mathbb{M}\xrightarrow{\ \ {}_HU\ \ } \text{Vect}_{\boldsymbol{k}}^{\text{f.d.}},\\ 
\omega _2 :  & \ \mathcal{C}_2 \cong 0\times  \mathcal{C}_2 \xrightarrow{\ \ F\ \ }  {}_H\mathbb{M}\xrightarrow{\ \ {}_HU\ \ } \text{Vect}_{\boldsymbol{k}}^{\text{f.d.}}
\end{align*}
are faithful and exact, there are bialgebras $A, B$ and $\boldsymbol{k}$-linear monoidal equivalences $G_1: \mathcal{C}_1\simeq \mathbb{M}^{A},\ G_2: \mathcal{C}_2\simeq \mathbb{M}^{B}$ such that $U^A\circ G_1=\omega _1,\ U^B\circ G_2=\omega _2$ by Theorem~\ref{4-3}. 
Then 
$$\mathcal{C}_1\times \mathcal{C}_2\simeq \mathbb{M}^{A}\times \mathbb{M}^{B}\cong \mathbb{M}^{A\oplus B}$$
as $\boldsymbol{k}$-linear monoidal categories. 
Thus we have $\boldsymbol{k}$-linear monoidal equivalence $G: \mathbb{M}^{H^{\ast}}\longrightarrow \mathbb{M}^{A\oplus B}$. 
This equivalence satisfies $U^{A\oplus B}\circ G=U^{H^{\ast}}$. 
So, by Corollary~\ref{3-6}, there is a weak bialgebra isomorphism  $g: A\oplus B\longrightarrow H^{\ast}$. 
Therefore, $H\cong H^{\ast\ast}\cong (A\oplus B)^{\ast}\cong A^{\ast}\oplus B^{\ast}$. 
This contradicts the fact that $H$ is indecomposable as a weak bialgebra. 
\end{proof}

\end{document}